\documentclass[oneside,english]{amsart}
\usepackage{textcomp}
\usepackage{amsthm}
\usepackage{amssymb}
\usepackage{amsmath}
\usepackage{amsfonts}
\usepackage{amscd}

\usepackage[final]{showkeys}
\usepackage{hyperref}

\usepackage[pdftex]{graphicx,color}

\makeatletter

\newcommand{\mc}{\mathcal}
\newcommand{\mf}{\mathfrak}
\newcommand{\C}{{\mathbb C}}
\newcommand{\Z}{{\mathbb Z}}
\newcommand{\R}{{\mathbb R}}
\newcommand{\N}{{\mathbb N}}
\newcommand{\U}{{\mathbb U}}
\renewcommand{\P}{{\mathbb P}}
\newcommand{\E}{{\mathbb E}}
\newcommand{\eps}{\varepsilon}
\newcommand{\Lap}{\Delta}
\newcommand{\vertiii}[1]{{\left\vert\kern-0.25ex\left\vert\kern-0.25ex\left\vert #1 
    \right\vert\kern-0.25ex\right\vert\kern-0.25ex\right\vert}}

\DeclareMathOperator{\Harm}{Harm}
\DeclareMathOperator{\Hom}{Hom}
\DeclareMathOperator{\Id}{Id}
\DeclareMathOperator{\Tr}{Tr}

\numberwithin{equation}{section}
\numberwithin{figure}{section}
  \theoremstyle{remark}
  \newtheorem*{rem*}{\protect\remarkname}
  \theoremstyle{plain}
  \newtheorem*{prop*}{\protect\propositionname}
  \theoremstyle{plain}
  \newtheorem*{lem*}{\protect\lemmaname}
\theoremstyle{plain}
\newtheorem{thm}{\protect\theoremname}
 \theoremstyle{definition}
 \newtheorem*{defn*}{\protect\definitionname}
  \theoremstyle{plain}
  \newtheorem{prop}[thm]{\protect\propositionname}
  \theoremstyle{remark}
  
  \theoremstyle{definition}
  
  \theoremstyle{plain}
  
  \theoremstyle{plain}
  \newtheorem{lem}[thm]{\protect\lemmaname}
  \theoremstyle{plain}
  \newtheorem*{thm*}{\protect\theoremname}

\makeatother

\usepackage{babel}
  \providecommand{\corollaryname}{Corollary}
  \providecommand{\definitionname}{Definition}
  \providecommand{\lemmaname}{Lemma}
  \providecommand{\propositionname}{Proposition}
  \providecommand{\remarkname}{Remark}
  \providecommand{\theoremname}{Theorem}
\providecommand{\theoremname}{Theorem}

\begin{document}

\title{Asymptotics of height change on toroidal Temperleyan dimer models}

\author{Julien Dub\'edat %
and Reza Gheissari}

\address{Department of Mathematics, Columbia University. 2990 Broadway, New
York, NY 10027, USA. }

\email{dubedat@math.columbia.edu \protect\footnote{Partially supported by NSF grant DMS-1005749} }

\address{Courant Institute of Mathematical Sciences, New York University, 251
Mercer St. New York, NY 10012 }

\email{reza@cims.nyu.edu}
\begin{abstract}
The dimer model is an exactly solvable model of planar statistical mechanics. In its critical phase, various aspects of its scaling limit are known to be described by the Gaussian free field. For periodic graphs, criticality is an algebraic condition on the spectral curve of the model, determined by the edge weights \cite{KOS}; isoradial graphs provide another class of critical dimer models, in which the edge weights are determined by the local geometry.

In the present article, we consider another class of graphs: general Temperleyan graphs, i.e. graphs arising in the (generalized) Temperley bijection between spanning trees and dimer models. Building in particular on Forman's formula and representations of Laplacian determinants in terms of Poisson operators, and 
under a minimal assumption - viz. that the underlying random walk converges to Brownian motion - we show that the natural topological observable on macroscopic tori converges in law to its universal limit, i.e. the law of the periods of the dimer height function converges to that of the periods of a compactified free field.

\tableofcontents{}
\end{abstract}
\maketitle

\section{Introduction}

The dimer model is a classical and extensively studied model of (planar) statistical mechanics, see e.g. the survey \cite{Ken_IAS}. Given an underlying graph ${\mc G}$, a dimer configuration (or perfect matching) is a subset of vertex-disjoint edges covering the graph; we consider here only the case where ${\mc G}$ is bipartite. For finite planar graphs, the model is exactly solvable, in the sense that its partition function can be represented as the determinant  of a modified adjacency matrix, the Kasteleyn matrix \cite{Kas_square}. More generally, for a graph on a genus $g$ surface, the partition function is a linear combination of $4^g$ such matrices, see \cite{Kas_square,Tesler,CimRes,CimResII}.

Exact solvability has allowed for a detailed analysis of large scale behavior of the dimer model. Following Thurston \cite{Thu_Con}, a dimer configuration on ${\mc G}$ may be mapped to a height function on ${\mc G}^*$; this gives a way to think about scaling limits of the dimer model (as the mesh of the graph goes to zero). When ${\mc G}$ is embedded on a torus, the height function is additively multivalued (i.e. picks up an additive constant when tracked along a non-contractible cycle on the torus) and can be decomposed into two components: an affine multivalued part (or instanton component, corresponding to the height change), and a  single-valued part (scalar fluctuation); see Section \ref{sssec:Dimers} or \cite{Dub_tors} for a more precise discussion.

For periodic graphs in the plane, \cite{KOS} shows the existence of three phases: a deterministic solid phase, an exponentially decorrelated gaseous phase, and a ``critical" liquid phase. We will focus here on that critical phase, where correlations have power law decay and the large-scale behavior of the height function is described by the Gaussian Free Field (GFF).

For critical dimer models, one expects the height fluctuations to be universal and described by a (compactified) free field. For specific lattices, Kenyon showed convergence of scalar fluctuations on the square lattice to the GFF in \cite{Ken_domino_GFF}, and Boutillier and de Tili\`ere established convergence of the height change on the hexagonal lattice in \cite{BdT_loop}. So far, universality results on dimers have focussed on two classes of graphs.

The first class consists of {\em periodic} graphs (where a microscopic finite fundamental domain is repeated many times to fill the plane or a macroscopic torus). The scalar fluctuations are shown to converge to a GFF in \cite{KOS} (for a proper choice of embedding); recently, \cite{KSW_torus} showed that the height change on tori also converges to its universal limit.

The second class consists of {\em isoradial} graphs (see \cite{Ken_isoradial}) derived from a lozenge tiling. These are not necessarily periodic but instead have a Yang-Baxter--type solvability. Universality for the height function in the plane is shown in \cite{dT_isoradial}. In the more restrictive set-up of so-called Temperleyan isoradial graphs, the full distribution of the height on tori is obtained in \cite{Dub_tors} (the scalar fluctuations and height change are asymptotically independent); see the recent \cite{Li_isoradial} for simply-connected domains.

A natural question is whether universality can be extended to other classes of graphs; a difficulty is that critical behavior is highly sensitive to edge weights, as exemplified by \cite{Chh_off}. In the present article, we undertake the analysis of dimers on {\em Temperleyan} graphs. This is a class of graphs ${\mc G}$ derived from a generic planar (or surface) graph $\Gamma$, in such a way that dimer configurations on ${\mc G}$ correspond to spanning trees on $\Gamma$ (generalized Temperley bijection, \cite{KPW}); ${\mc G}$ is obtained by superimposing $\Gamma$ with its dual $\Gamma^*$. Remark that these three classes of graphs are not disjoint (e.g. there are graphs which are periodic, Temperleyan and not isoradial etc.).

Since spanning trees can be derived from random walks (Wilson's algorithm, \cite{Wi}), a natural (and essentially minimal) criticality condition is that the underlying random walk (RW) on $\Gamma$ converges, up to time change, to Brownian motion (BM). Under this assumption, Yadin and Yehudayoff \cite{yadin_lerw} showed convergence of the Loop-Erased Random Walk (LERW) to SLE$_2$ (Schramm-Loewner Evolution with $\kappa=2$), extending celebrated work of Lawler, Schramm, Werner on regular lattices \cite{LSW_LERW}.

Since dimers on ${\mc G}$ correspond to Uniform Spanning Trees (USTs) on $\Gamma$, and USTs can be generated from LERWs (\cite{Wi}), this suggests universality of dimers when the underlying RW converges to BM. Remark also that the dimer height can be expressed in terms of windings of LERWs (\cite{KPW}). However at this stage it is unclear how to implement rigorously this heuristic.

In the present article, we thus focus on the height change of dimers on a torus, when the underlying random walk converges to BM; our main result, Theorem \ref{Thm:discrg}, establishes that this quantity is indeed asymptotically universal. The limiting distribution is consistent with the conjectured compactified free field invariance principle, see e.g. the discussion following Theorem \ref{Thm:discrg} or \cite{Dub_abelian}. The method is based not on Kasteleyn enumeration, but on Temperley's bijection and Forman's formula \cite{Forman_graph}, a deformation of the Matrix-Tree Theorem enumerating spanning trees in terms of Laplacian determinants. 

As already mentioned, there is some overlap with earlier results on the height change of dimers on tori, although the assumptions and techniques presented here are quite distinct. The case of the hexagonal lattice (on tori with purely imaginary modulus) was treated in \cite{BdT_loop}. Temperleyan isoradial dimers were analyzed in \cite{Dub_tors}, where the relation to the compactified free field is also pointed out. Universality for cycle-rooted spanning forests on surfaces (more precisely, the topology of their roots) was obtained by Kassel and Kenyon in \cite{KasKen}, under a rather strict convergence condition (roughly speaking, $C^1$ convergence of the Green kernel; we work here under \eqref{eq:weakconv}, which does not even guarantee its $C^0$ convergence). Combining ideas of \cite{BdT_loop} and \cite{KOS}, \cite{KSW_torus} described the limiting distribution of the height change in the periodic case.

Some components of the argument subsist in higher genus, in particular the universality of ratios of Laplacian determinants (Proposition \ref{Prop:detconv}). This in turn gives information on the homology of root cycles of cycle-rooted spanning forests. However, the topology of these cycles and their relation with dimers is more complicated than in genus 1, and it is unclear whether there is a concise description of a limiting topological observable in that case. Remark also that more information can be extracted from Kenyon's vector bundle Laplacian \cite{Ken_VBL}.

The manuscript is organized as follows. Section 2 presents the relevant objects and provides background. The main results are stated in Section 3. Universality of the limit is established in Section 4; explicit computations for the square lattice are provided in Section 5, and the limiting distribution is identified in Section 6.\\

{\bf Acknowledgments.} It is our pleasure to thank anonymous referees for their detailed and insightful comments.

\section{Set-up and notation}

\subsection{Basic structures}

\subsubsection{Underlying graphs}

We start from a finite graph $\Gamma=(V,E)$ embedded on a torus $\Sigma\stackrel{\rm def}=\C/(\Z+\tau\Z)$, where the modulus $\tau$ of the torus is in the upper half-plane: $\Im\tau>0$. The embedded edges do not cross (except at their endpoints) and bound faces homeomorphic to disks.

The graph $\Gamma$ is endowed with nonnegative edge weights (conductances) on oriented
edges ($c(xy)\geq 0$, where $(xy)$ is an oriented edge); we do {\em not} assume $c(xy)=c(yx)$. For example, non-symmetric weights allow to treat the case of dimers of the hexagonal lattice, as shown in \cite{KPW}; and of some models of random walks in random environments (e.g. for balanced environments).

Let $\Gamma^*=(V^*,E^*)$ be the dual graph of $\Gamma$, so that faces of $\Gamma$ are in bijection with vertices of $\Gamma^*$ and vice versa: two vertices in $\Gamma^*$ are adjacent if they correspond to faces of $\Gamma$ sharing a common boundary edge. If $e\in E$ is an edge separating the faces $f$ and $f'$ of $\Gamma$, then $e^*=(ff')\in E^*$ denotes its dual edge. All edges on $\Gamma^*$ are assigned weight~1.

Finally we consider the graph ${\mc G}$ obtained by superimposing $\Gamma$ and $\Gamma^*$. More precisely, ${\mc G}=(V_{{\mc G}},E_{{\mc G}})$ is the bipartite graph whose black vertices $V_{\mc G}^B$ are in bijection with $V\sqcup V^*$ and white vertices $V_{\mc G}^W$ are in bijection with $E^*\simeq E$. See Figure \ref{Fig:tempgraph}.
\begin{figure}[htb]
\begin{center}
\leavevmode
\includegraphics[width=0.8\textwidth]{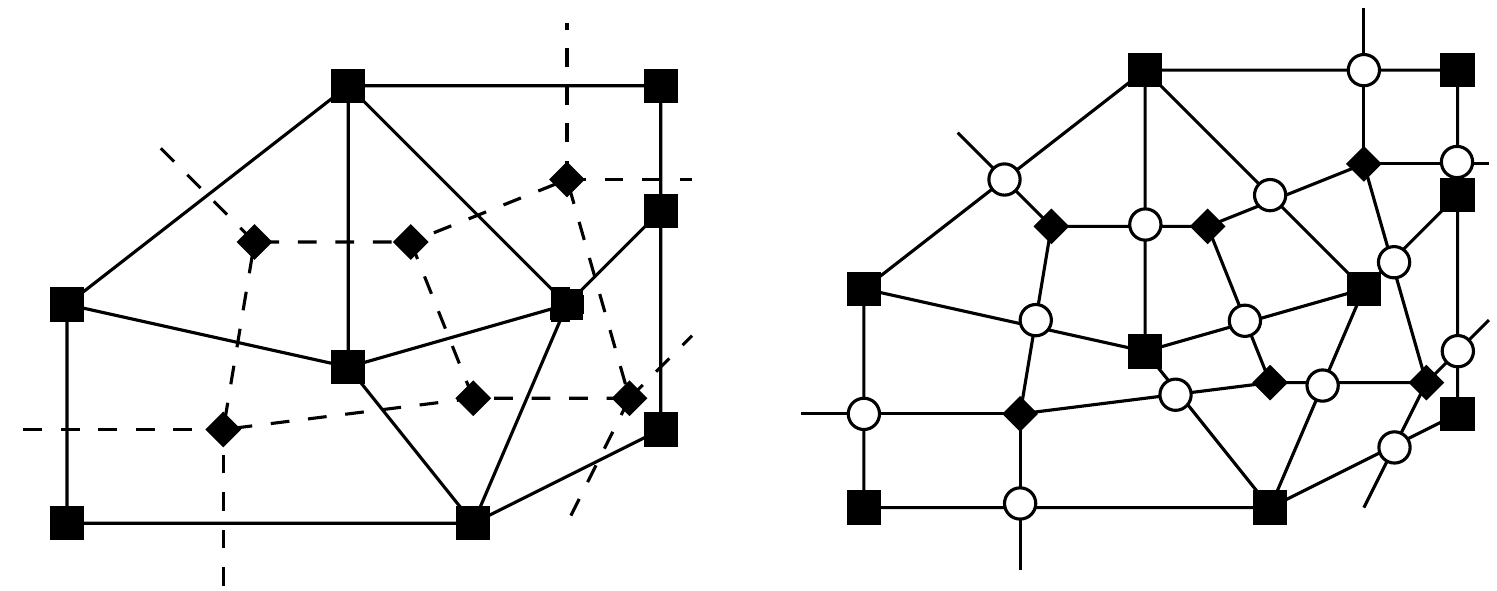}
\end{center}
\caption{Left: portion of $\Gamma$ (solid) and its dual $\Gamma^*$ (dashed). Right: bipartite graph ${\mc G}$.}
\label{Fig:tempgraph}
\end{figure}

Edges in ${\mc G}$ are of two types: $(v,(vv'))$ and $(f,(ff'))$ where $v,v'$ are adjacent in $\Gamma$ and $f,f'$ are adjacent in $\Gamma^*$. In other words, each edge in ${\mc G}$ corresponds to an oriented edge in $\Gamma$ or $\Gamma^*$. Notice that ${\mc G}$ is a quadrangulation, i.e. all its faces have degree 4. We call such a graph ${\mc G}$ {\em Temperleyan} \cite{KPW}.

One may associate weights to edges of ${\mc G}$ in the following way:
\begin{align*}
w(v,(vv'))&=c(vv')&{v,v'\in V}, v\sim v',\\
w(f,(ff'))&=1&{f,f'\in V^*}, f\sim f'.
\end{align*}
Notice than some edges may have zero weight.

Remark that by Euler's formula for $\Gamma$,
$$|V_{\mc G}^B|-|V_{\mc G}^W|=|V|-|E|+|V^*|=2-2g$$
where $g$ is the genus of $\Sigma$; in particular ${\mc G}$ is balanced iff $\Sigma$ is a torus. %

\subsubsection{Spanning forests}

A {\em vector field} (without zeroes) on $\Gamma$, in the sense of Forman, is an assigment $Y:V\rightarrow V$ s.t. $v\sim Y(v)$ for all $v$. Since $\Gamma$ is finite, $(Y^k(v))_{k\geq 0}$ is eventually periodic for any $v$. It is then easy to see that the data of a vector field is equivalent to that of an oriented {\em cycle-rooted spanning forest} (CRSF), in the sense of Kenyon, viz. an oriented spanning subgraph in which each connected component contains exactly one cycle; other edges are oriented towards the unique cycle (``root") in their connected component. A CRSF is {\em incompressible} \cite{Ken_VBL} if all its root cycles are noncontractible on $\Sigma$. Unless mention of the contrary, we will only consider CRSFs which are {\em oriented} and {\em incompressible}.

Given a CRSF $F$ on $\Gamma$, one can construct a dual CRSF $F^*$ on $\Gamma^*$, s.t. each $w\in V^W_{\mc G}\simeq E\simeq E^*$ is crossed either by an edge of $F$ or of $F^*$. If $F$ has $k$ root cycles, they are in the same homology class, up to sign; $F^*$ also has $k$ root cycles in the same class and is specified by $F$ up to the orientation of these $k$ cycles, which may be chosen freely. (This property is specific to incompressible CRSFs on tori).

\subsubsection{Dimers}\label{sssec:Dimers}

A {\em dimer configuration} or {\em perfect matching} on ${\mc G}$ is a subset ${\mf m}$ of edges of ${\mc G}$ s.t. each vertex in $V_{\mc G}$ is an endpoint of exactly one edge in  ${\mf m}$. The weight of a dimer configuration is the product of the weights of edges it contains:
$$w({\mf m})=\prod_{e\in{\mf m}}w(e).$$
There is then a natural probability measure on the space of dimer configurations, viz. the Boltzman measure given by
\begin{align*}
\mu\{{\mf m}\}=  \frac{\prod_{e\in{\mf m}}w(e)}{{\mc Z}}
\end{align*}
where the {\em partition function} is 
\begin{align*}
{\mc Z}= & \sum_{{\mf m}}\prod_{e\in {\mf m}}w(e).
\end{align*}

Following Thurston \cite{Thu_Con}, one can map a dimer configuration on ${\mc G}$ to a \emph{height function} on the dual graph ${\mc G}^*$. The height function is defined as follows (see \cite{Ken_IAS}): Given a perfect matching
${\mf m}$, define a 1-form $\omega_{{\mf m}}$ (i.e. an antisymmetric function on oriented edges) on the dual graph ${\mc G}^*$
given by 
\begin{align*}
\omega_{{\mf m}}((bw)^*)= & \left\{ \begin{array}{cc}
1 & \mbox{if b and w are matched}\\
0 & \mbox{otherwise.}
\end{array}\right.
\end{align*}
Here $(bw)$ denotes an oriented edge of ${\mc G}$ - oriented from black to white - and $(bw)^*$ is the edge of ${\mc G}^*$ dual to $(bw)$, oriented so that $((bw),(bw)^*)$ is a direct frame.

Define $(d\omega_{\mf m})(f)$ as the sum of $\omega_{{\mf m}}(e)$ over edges $e$ bounding the face $f$ of ${\mc G}^*$, taken counterclockwise; recall that faces of ${\mc G}^*$ correspond to vertices of ${\mc G}$. It is straightforward to check that $d\omega_{{\mf m}}(f)$ is $1$ (resp. $-1$) if $f$ corresponds to black (resp. white) vertex of ${\mc G}$.

Consider a reference matching ${\mf m}_0$.
Let $\omega_0$ be the corresponding $1$-form. Then $d(\omega-\omega_0)=0$.

Consequently, locally we can write $\omega_{{\mf m}}-\omega_0=dh$ (i.e. $(\omega_{{\mf m}}-\omega_0)(vv')=h(v')-h(v)$), where $h$ is a {\em height function} defined on ${\mc G}^*$ (given up to an additive constant). 

Because the torus has non-trivial homology, $h$ need not be defined as a function on all of ${\mc G}^*$, but merely as an additively multivalued function (i.e. picking an additive constant when traced along a non-contractible cycle). Concretely, one may lift ${\mc G}^*$ and the matchings to the universal cover $\C$ of $\Sigma$, and then $h$ defines an additively quasi-periodic function in the following sense: there are $a,b$ such that for all vertex $v$ 
in the lift of ${\mc G}^*$ and $m,n\in\Z$, 
$$h(v+m+n\tau)=h(v)+ma+nb.$$

A standard basis of the homology group $H_1(\Sigma,\Z)$ is given by the cycle $A:t\mapsto t$ and the cycle $B:t\mapsto t\tau$, $t\in\R/\Z$.

Of particular interest to us will be the {\em periods} $(a,b)=(\int_A dh,\int_B dh)$, where
\begin{equation}\label{eq:period}
\int_A dh=\sum_{e\in\gamma}(\omega_{{\mf m}}-\omega_0)(e)
\end{equation}
where $\gamma$ is a lattice cycle on ${\mc G}^*$ homotopic to $A$ (and similarly for $B$ or any element of $H_1(\Sigma,\Z)$). These do not depend on the choice of $\gamma$ (since $d(\omega_{{\mf m}}-\omega_0)=0$); and, by construction, are integers.

The construction explained above is valid for general bipartite graphs. In the present context of Temperleyan graphs, one can circumvent the use of a reference matching ${\mf m}_0$ (or form $\omega_0$) as follows. 

Instead of considering any cycle $\gamma$ running on ${\mc G}^*$, we consider only cycles running to the immediate lefthand side of a cycle on the primary graph $\Gamma$. One verifies that deforming $\gamma$ to $\gamma'$, a cycle homotopic to $\gamma$ satisfying the same conditions, involves crossing the same number of black and white vertices, so that $\int_\gamma\omega$ is independent of the choice of $\gamma$.

In conclusion, a dimer configuration ${\mf m}$ on ${\mc G}^*$ defines an element of $H_1(\Sigma,\Z)^{\vee}$, via
\begin{equation}\label{eq:hperiod2}
[\gamma]\in H_1(\Sigma,\Z)\longmapsto\int_\gamma\omega_{{\mf m}}
\end{equation}
where $M^\vee=\Hom(M,\Z)$ denotes the dual of the $\Z$-module $M$.

\subsubsection{Temperley's bijection}

There is a one-to-one correspondence between pairs $(F,F^*)$ of dual CRSFs on $\Gamma,\Gamma^*$ and perfect matchings ${\mf m}$ on ${\mc G}$, the generalized Temperley's bijection \cite{KPW} (which we present here in the toroidal setting). Set
 \begin{align*}
(v,(vv'))\in{\mf m}\Leftrightarrow (vv')\in F&&{v,v'\in V}, v\sim v'\\
(f,(ff'))\in{\mf m}\Leftrightarrow (ff')\in F^*&&{f,f'\in V^*}, f\sim f'.
\end{align*}
The weight of a CRSF (resp. of a perfect matching) is the product of the weights of edges it contains. Recall that edge weights on $\Gamma^*,{\mc G}$ are derived from those on $\Gamma$. Then Temperley's bijection is also weight-preserving \cite{KPW}. See Figure \ref{Fig:tempbij}.
\begin{figure}[htb]
\begin{center}
\leavevmode
\includegraphics[width=1.0\textwidth]{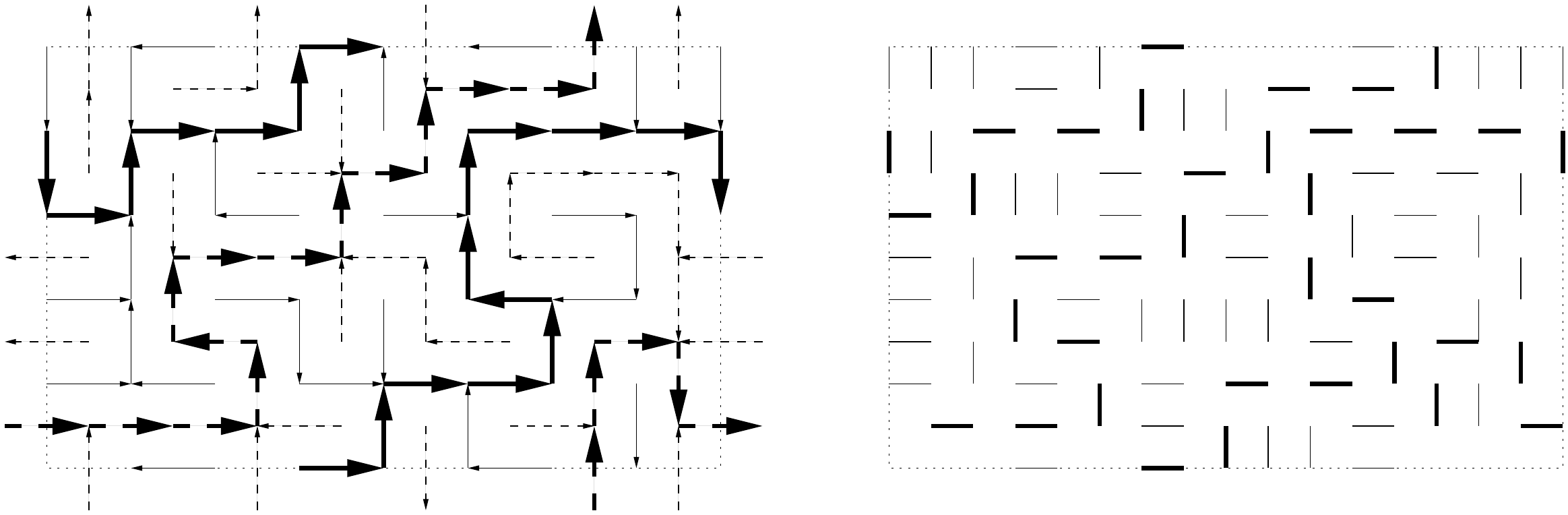}
\end{center}
\caption{Left: a pair of dual CRSFs (solid: primal forest; dashed: dual forest; thick: root cycles; dotted: fundamental domain). Right: corresponding dimer configuration (thick: dimers on the root cycles). Here $[{\mf m}]=[A]+[B]$.}
\label{Fig:tempbij}
\end{figure}

Let us discuss the relation between the height periods defined in \eqref{eq:period} and the pair of CRSFs $(F,F^*)$ corresponding to the matching ${\mf m}$.

Let $[F]+[F^*]\in H_1(\Sigma,\Z)$ be its homology class, defined as the sum of the homology classes of the oriented root cycles of $F,F^*$. Trivially this is an even class (i.e. in $2H_1(\Sigma,\Z)$), and we define 
\begin{equation}\label{eq:cyclehomol}
[{\mf m}]=\frac 12\left([F]+[F^*]\right)\in H_1(\Sigma,\Z).
\end{equation}
Hence, given ${\mf m}$ and the corresponding pair $(F,F^*)$ we defined an element of $H_1(\Sigma,\Z)^\vee$ in \eqref{eq:hperiod2} and an element of $H_1(\Sigma,\Z)$. 

There is a canonical identification $H_1(\Sigma,\Z)^\vee\simeq H_1(\Sigma,\Z)$ via the intersection pairing $\cdot$, the antisymmetric bilinear form such that $[A]\cdot [B]=1$ ($[\gamma_1]\cdot[\gamma_2]$ counts the intersections of $\gamma_1$ and $\gamma_2$ with a sign depending on orientation, see e.g. III.1 in \cite{FarKra}).

Without loss of generality, we may assume that the root cycles are homotopic to $\pm A$, so that $[\mf m]=k[A]$, $k\in\Z$. By choosing a representative of $[A]$ running on ${\mc G}^*$ along a root of $F$, one gets immediately $\int_A \omega_{\mf m}=0$. Then one chooses a representative of $[B]$ running along branches of $F$. A more careful examination shows that $\int_B\omega_{\mf m}=k$. 

In conclusion we have
$$\int_\gamma \omega_{\mf m}=[{\mf m}]\cdot [\gamma]$$ 
so that the points of view of \eqref{eq:hperiod2} and \eqref{eq:cyclehomol} are identified via the intersection pairing.

\subsection{Convergence}

The discussion is so far purely combinatorial, and is for instance invariant under deformation of the embedding. We now introduce an essentially minimal condition that guarantees criticality of a sequence of Temperleyan graphs. We employ ``criticality" in the somewhat loose sense that aspects of the fine mesh limit are described by conformally invariant objects, as e.g. in \cite{Ken_isoradial}. 

Recall that the oriented edges of $\Gamma=(V,E)$ are assigned non-negative weights (conductances). Given these weights, one can introduce a Laplacian $\Delta:\R^{V}\rightarrow\R^V$ %
by setting
$$(\Delta f)(x)=\sum_{y\in V: y\sim x}c(xy)(f(x)-f(y))$$
where $\sim$ denotes adjacency (using the positive Laplacian convention). Then $-\Delta$ is the generator of a continuous-time random walk $(X_t)_{t\geq 0}$, where the conductances give the jump rates. 

To avoid trivialities, we assume that the graph is irreducible in the sense that any two vertices are joined by a chain of edges with positive weights. 

Alternatively, one may consider its discrete-time skeleton, the Markov chain $(X_n)_{n\in\N}$ specified by 
$$\P(X_{n+1}=y|X_n=x)=\frac{c(xy)}{\sum_{y'\sim x}c(xy')}.$$ 
The objects of interest to us - harmonic functions, harmonic measures, Poisson operators, etc. - are invariant under time change, so that the distinction between discrete and continuous time is essentially moot.

The mesh size is
defined as $\delta:=\sup_{x,y\in V:x\sim y}|x-y|$. We will consider a sequence of graphs
$(\Gamma_\delta)_{\delta>0}$ indexed by mesh size, where $\delta$ goes to zero along some sequence. The dependence of various objects on $\delta$ will be omitted when there is no ambiguity.

The key (and essentially minimal) assumption is that, as the mesh goes to zero, the corresponding random walk converges weakly, up to time change, to Brownian motion on $\Sigma$. We do not expect that there is a natural and more general condition for Temperleyan graphs. Among several equivalent specific formulations, we choose the following (essentially as in \cite{yadin_lerw}). 

Let $\tilde X_\delta$ be the continuous process obtained from the random walk $X_\delta$ on $\Gamma_\delta$ by linear interpolation between jump times. Let $x_\delta\in V_\delta$, the starting vertex, be such that $x_\delta\rightarrow 0$. Let $\sigma$ be the time of first exit of, say, $[-1,1]+\tau[-1,1]$ (the union of four fundamental domains; any set containing a fundamental domain in its interior would work). Let $\mu_\delta$ be the law induced on $C([0,1],\C)$ by $t\mapsto \tilde X_\delta(t\sigma)$. The state space $C([0,1],\C)$ is metrized by uniform convergence up to time reparametrization, viz. 
$$d(\gamma_1,\gamma_2)=\inf_{\phi:[0,1]\rightarrow [0,1]}\|\gamma_1-\gamma_2\circ\phi\|_\infty$$
where the infimum is taken over increasing homeomorphisms $\phi:[0,1]\rightarrow[0,1]$. This turns $C([0,1],\C)$ (quotiented by reparametrization) into a Polish space. 
Now let $\mu$ be the law induced by $(B_{t\sigma})_{0\leq t\leq 1}$, where $B$ is a standard Brownian motion. The condition is then simply that 
\begin{equation}\label{eq:weakconv}
\mu_\delta{\rm\ converges\ to\ }\mu{\rm\ weakly.} 
\end{equation}

This condition is for instance satisfied for simple random walk on supercritical percolation clusters \cite{BerBis}, among other models of random walks in random environments. Recall also from \cite{KPW} that dimers on the hexagonal lattice correspond to a periodic, non-reversible random walk; our arguments do not rely on any specific properties of reversible random walks.

Using planarity, one may check that the condition \eqref{eq:weakconv} is independent of the choice of starting point (with probability bounded away from 0, a random walk started from $x$ completes a small but macroscopic loop around $y$; on that event we can couple the RW started from $y$ and the RW started from $x$ after completing that loop). We denote by $\P_\delta^x$ the law of the RW on $\Gamma_\delta$, simply denoted by $X$, started from (the point on $V(\Gamma_\delta)$ closest to) $x$. Similarly, $\P^x$ denotes the law of standard BM started from $x$.

For our purposes, a key consequence of the weak convergence condition is the convergence of harmonic measure. For $D\subset\Sigma$ an open set (with Jordan boundary, say), $J\subset\partial D$ and $x\in D$, we denote
\begin{align*}
\Harm_\delta(x,J,D)&=\P_\delta^x(X{\rm\ exits\ }D{\rm\ on\ }J),\\
\Harm(x,J,D)&=\P^x(B{\rm\ exits\ }D{\rm\ on\ }J).
\end{align*}

The following is a weaker version of Lemma 4.8 in \cite{yadin_lerw}. 

\begin{lem}\label{Lem:harmmeas}
Fix $\eta\in (0,1)$ and $\eps>0$. There is $\delta_0$ such that for any $\delta\leq \delta_0$, $x\in\C$, any simply connected Jordan domain $D$ with $B(x,\eta)\subset D\subset B(x,\eta^{-1})$ and arc $J\subset \partial D$, we have 
$$|\Harm_\delta(x,J,D)-\Harm(x,J,D)|\leq\eps.$$
\end{lem}

Remark that \eqref{eq:weakconv} does not guarantee the type of $C^1$ convergence of Green's functions which has been instrumental in the asymptotic analysis of dimers since~\cite{Ken_domino_conformal}.

\subsection{Twisted Laplacians}

A graph $\Gamma=(V,E)$ on $\Sigma=\C/(\Z+\tau\Z)$ lifts to a graph $\tilde\Gamma=(\tilde V,\tilde E)$ on the universal cover $\C$, and the Laplacian on $\Gamma$ lifts to an operator $\tilde\Lap:\C^{\tilde V}\rightarrow\C^{\tilde V}$. 

The fundamental group operates on $\C^{\tilde V}$ by translations, and that action commutes with $\tilde\Lap$. Consequently $\tilde\Lap$ stabilizes eigenspaces of translations.

Fix a character $\chi:\pi_1(\Sigma)\rightarrow\U$, $\U$ the unit circle (equivalently, choose $\chi(A)$ and $\chi(B)$  in $\U$), and set
$$(\C^{\tilde V})_\chi=\{\phi\in\C^{\tilde V}:\forall\gamma\in\pi_1(\Sigma), \phi(\cdot+\gamma)=\chi(\gamma)\phi(\cdot) \}.$$
We denote by $\Lap_\chi:(\C^{\tilde V})_\chi\rightarrow (\C^{\tilde V})_\chi$ the restriction of $\tilde\Lap$ to $(\C^{\tilde V})_\chi$. Remark that for the trivial character $\chi=1$, there is a natural identification $\Lap_1\simeq\Lap:\C^V\rightarrow\C^V$.

An equivalent formulation \cite{Ken_VBL} identifies $\Lap_\chi$ with the Laplacian operating on the flat line bundle with monodromy given by the character $\chi$.

Concretely, one can fix a fundamental domain for $\Sigma$ in $\C$ corresponding to an $A$- and $B$- cycle (simple paths on the dual graph to $\tilde\Gamma$). For $v\in\tilde V$ in this fundamental domain, set $e_v(v')=\delta_{v,v'}\chi(v'-v)$. This gives a basis of $(\C^{\tilde V})_\chi$. Relative to this basis, the matrix of $\Lap_\chi$ differs from the matrix of $\Lap$ on entries corresponding to edges crossing the cycles. More precisely,
$$(\Lap_\chi)(v,v')=-c(v,v')\chi^\pm(A)$$
if $(vv')$ crosses $A$, with $\pm=1$ if $(A,(vv'))$ is directly oriented and $-1$ otherwise; the corresponding statement holds for the $B$-cycle, and all other entries of the matrix are unchanged.

From the maximum principle it is easy to see that $\Lap_\chi$ has trivial kernel (equivalently, $\det(\Lap_\chi)\neq 0$) iff $\chi$ is non-trivial.

The Matrix-Tree theorem states that the determinant (more precisely, a cofactor) of the discrete Laplacian enumerates (weighted) spanning trees on $\Gamma$. The following is a twisted version of the  Matrix-Tree Theorem, due to Forman \cite{Forman_graph} and, independently, Kenyon \cite{Ken_VBL}.

\begin{thm}\label{Thm:Forman}
Let $\Gamma=(V,E,c)$ be a weighted graph on the torus $\Sigma$, $\chi:\pi_1(\Sigma)\rightarrow\U$ a character, and $\Lap_\chi$ the Laplacian on $\Gamma$. Then
$$\det(\Lap_\chi)=\sum_{F{\rm\ CRSF}}\left(\prod_{e\in F}c(e)\prod_{\gamma{\rm\ cycle\ in\ }F}(1-\chi(\gamma))\right).$$
\end{thm}
In the reversible case ($c(vv')=c(v'v)$ for all $v,v'$), this becomes
$$\sum_{F{\rm\ CRSF}}\left(\prod_{e\in F}c(e)\prod_{\Gamma{\rm\ cycle\ in\ }F}(2-\chi(\gamma)-\chi(\gamma)^{-1}))\right)$$
where the sum is over unoriented CRSFs. There are also extensions of this result~\cite{Ken_VBL}, which will not be of use here.

\section{Statement of Main Results}

The goal is to show convergence of the distribution of $[{\mf m}]$ to an explicit, universal limit under the general assumption that the underlying random walk on $\Gamma_\delta$ converges to Brownian motion \eqref{eq:weakconv} as the mesh goes to zero (along some sequence). The argument has three essentially independent components.

The assumption \eqref{eq:weakconv} is essentially a condition on convergence of harmonic measure (Lemma \ref{Lem:harmmeas}). Our first task is to show that ratios of Laplacian determinants can be expressed in terms of harmonic measure and are consequently universal in the small mesh limit, resulting in 

\begin{prop}\label{Prop:detconv}
Let $\chi,\chi':\pi_1(\Sigma)\rightarrow\U$ be two non-trivial characters. Then
$$\frac{\det(\Lap^{\delta}_{\chi'})}{\det(\Lap^{\delta}_\chi)}\stackrel{\delta\searrow 0}{\longrightarrow}\frac{g(\chi')}{g(\chi)}$$
where $g$ is a function which does not depend on the sequence $(\Gamma_\delta)$ satisfying \eqref{eq:weakconv}.
\end{prop}

Given the universality result of Proposition \ref{Prop:detconv}, in order to identify the righthand side we can work on graphs of our choosing. We pick square lattices, for which we have not only convergence in the weak sense of \eqref{eq:weakconv} (up to time change), but also precise heat kernel asymptotics. This leads to

\begin{prop}\label{Prop:square}
Fix $\chi,\chi'$ non-trivial unitary characters of $\pi_1(\Sigma)$ and take $\Gamma_\delta=\delta\Z^2/(\Z+\tau_\delta\Z)$ where $\tau_\delta\in\delta\Z^2$ and $\tau_\delta=\tau+o(1)$, $\delta^{-1}\in\N$. Then 
$$\frac{\det(\Lap_{\chi'}^{\delta})}{\det(\Lap_\chi^{\delta})}\stackrel{\delta\searrow 0}{\longrightarrow}\frac{T(\chi')^2}{T(\chi)^2}$$
where
$$T(\chi)=|\eta(\tau)^{-1}e^{i\pi v^2\tau}\theta(u-v\tau|\tau)|$$
and $\chi(m\tau+n)=\exp\left(2i\pi(mu+nv)\right)$.
\end{prop} 
Here $\eta$ denotes the Dedekind $\eta$ function \eqref{eq:eta} and $\theta$ the odd $\theta$ function \eqref{eq:theta}, with conventions as in \cite{Chandra}.

On the square lattice, we take all (nearest neighbor) conductances to be equal, so that the corresponding random walk is a simple random walk (see Section \ref{Sec:square}  for definitions and conventions for $\theta$ functions).

Finally, we need to identify the limiting distribution of $[{\mf m}]$ from the limit of Laplacian determinants; this is performed in 

\begin{thm}\label{Thm:discrg}
As the mesh $\delta$ goes to zero, the law induced by $[{\mf m}]$ on $H_1(\Sigma,\Z)\simeq\Z+\tau\Z$ converges to the discrete Gaussian distribution $\P_0$ specified by
$$\P_0\{r\tau+s\}\propto \exp\left(-\frac\pi {2\Im\tau} |r\tau+s|^2\right)$$
for $r,s\in\Z$.
\end{thm}

There is a natural interpretation of the limiting distribution in terms of the compactified free field (e.g. Section 2.1.3 in \cite{Dub_tors}), which is already known to be the limit of the height on Temperleyan isoradial graphs (see \cite{Dub_tors}). The compactified free field on $\Sigma$ (with compactification radius $1$ and coupling constant $g_0$) is formally a measure on fields $\phi:\Sigma\rightarrow\R/2\pi \Z$ given by
$$\exp\left(-\frac{g_0}{4\pi}\int_\Sigma|\nabla\phi|^2dA\right){\mc D}\phi$$
which has the following well-defined interpretation: $\phi$ is the sum of two independent components, the scalar component $\phi_s$ - a Gaussian free field on $\Sigma$ - and the instanton component $\phi_h$, a harmonic, additively multivalued function on $\Sigma$ with periods in $2\pi\Z$. 

If $(\int_Ad\phi_h,\int_Bd\phi_h)=2\pi(-r,s)$, then $\phi_h$ lifted to $\C$ is the affine function
$$z\mapsto 2\pi\left(-r\left(\Re(z)-\frac{\Re\tau}{\Im\tau}\Im(z)\right)+s\frac{\Im(z)}{\Im\tau}\right)=2\pi\frac{\Im\left(z(r\bar\tau+s)\right)}{\Im\tau}$$
with Dirichlet energy
$${\rm Area}(\Sigma)(2\pi)^2\left(\frac{|r\bar\tau+s|}{\Im\tau}\right)^2=4\pi^2\frac{|r\tau+s|^2}{\Im\tau}.$$
The distribution of $\phi_h$ is specified by 
$$\P\left(\int_Ad\phi_h=-2\pi r, \int_Bd\phi_h=2\pi s\right)\propto\exp\left(-g_0\pi\frac{|r\tau+s|^2}{\Im\tau}\right)
$$
for $r,s\in\Z$. Remark that we take this period distribution as part of the rigorous definition of a compactified free field (rather than a property of it). In terms of dimer height, $[{\mf m}]=r\tau+s$ corresponds to a height variation of $-r$ (resp. $s$) along the $A$ (resp. $B$ cycle). 

In the Temperleyan isoradial case \cite{Dub_tors}, $2\pi h$ converges to a compactified free field with $g_0=\frac 12$. We conjecture this is also the case for general Temperleyan graphs under \eqref{eq:weakconv}. Theorem \ref{Thm:discrg} verifies this conjecture for the instanton component. 


%
%
%
%
%
%
%
%
%
%
%
%
%
%
%
%
%
%
%
%
%
%
%
%
%
%
%
%
%
%
%
%
%
%
%
%
%
%
%
%
%
%
%
%
%
%
%
%
%
%
%

%
%
%
%
%
%
%
%
%
%

%

\section{Convergence of Laplacian determinants}

Our goal in this section is to prove Proposition \ref{Prop:detconv}, i.e. that ratios of Laplacian determinants have a universal limit under \eqref{eq:weakconv}. We start with some intuitive justification of that fact in terms of loop measures. Then we proceed to express these ratios in terms of series built from harmonic measures. Uniform convergence of these series is justified by a contraction argument, and term-wise convergence follows from Lemma \ref{Lem:harmmeas}, which proves Proposition \ref{Prop:detconv}. 

The decompositions considered here are rather similar to those employed in \cite{Dub_SLEGFF} in simply-connected domains; new features of this section include twists by characters and working under the weak convergence assumption \eqref{eq:weakconv}. In \cite{Dub_tors}, a result similar to Proposition \ref{Prop:detconv} is proved for Kasteleyn operators (rather than Laplacians) on Temperleyan isoradial graphs; the arguments there are based on character interpolation and are essentially disjoint from those presented in this section.

\subsection{Interpretation in terms of loop measure}

Let $\chi,\chi':\pi_1(\Sigma)\rightarrow\U$ be two nontrivial characters s.t. $\chi(A)=\chi'(A)$ and $\chi(B)\neq\chi'(B)$. Our goal is to show that 
\begin{equation}\label{eq:Lapratio}
\frac{\det(\Lap_{\chi'})}{\det(\Lap_{\chi})}
\end{equation}
is universal in the small mesh limit. 

We start with an informal argument, aimed at the reader familiar with loop measures \cite{LW,LJ_loops}. 

First we observe that the ratio \eqref{eq:Lapratio} is unchanged when replacing the conductances $c(xy)$ by the normalized conductances:
$$c'(xy)=\frac{c(xy)}{\sum_{y':y'\sim x}c(xy')}$$
so that we may assume $\sum_{y:y\sim x}c(xy)=1$ for all $x\in\Gamma$. Then $\Lap=\Id-P$ where $P$ is the Markov chain transition matrix for the discrete-time random walk on $\Gamma$. 

Consider the space of rooted loops on $\Gamma$, i.e. paths on $\Gamma$ of the type 
$$\ell=(v_0,v_1,\dots,v_{n-1},v_{n}=v_0)$$ 
for arbitrary length $|\ell|\stackrel{\rm def}=n$. This space is endowed with the (rooted) loop measure $\mu^{loop}$ specified by
$$\mu^{loop}((v_0,\dots,v_n=v_0))=\frac{1}{n}c(v_0v_1)\dots c(v_{n-1}v_0).$$
Remark that a loop $\ell$ induces a homology class $[\ell]\in H_1(\Sigma,\Z)$. One possible approach to universality of the ratio \eqref{eq:Lapratio} is based on the following formal identity
$$``\frac{\det(\Lap_{\chi'})}{\det(\Lap_{\chi})}=\exp\left(-\int (\chi'([\ell])-\chi([\ell]))d\mu^{loop}(\ell)\right)".$$
A difficulty is that the integral is not absolutely convergent (due to long loops); this can be remediated by introducing a small, positive killing rate and letting it go to zero (this gives a meaning to the RHS and it can then be checked that it equals the LHS). Working on bridge measures seems also inconvenient under \eqref{eq:weakconv}. 

Since weak convergence condition is essentially a condition on convergence of harmonic measure, we work instead with somewhat less probabilistic but more robust Poisson operators. The general idea is to enumerate loops as they travel back and forth between macroscopically distant cycles on the torus.

\subsection{Poisson operators}

For now we work on a fixed graph $\Gamma$ with mesh $\delta$. 
Let $\gamma_1$ and $\gamma_2$ be two disjoint simple cycles on $\Gamma$ homotopic to the $A$-cycle and denote $\gamma=\gamma_1\sqcup\gamma_2$ (for levity, we omit the dependence on $\delta$ when there is no ambiguity). In the small mesh limit, one may think of $\gamma_1$ and $\gamma_2$ as $[0,1]$ and $\frac{\tau}2+[0,1]$ up to $o(1)$. We define Poisson operators $R_\chi^\delta,Q_\chi^\delta$ by
\begin{equation}\label{eq:Rchidef}
\begin{split}
R_\chi^\delta:(\C^{\gamma_1})_\chi&\longrightarrow (\C^V)_\chi\\
f&\longmapsto \left(x\mapsto (R_\chi ^\delta f)(x)=\E^x_\delta(f(X_{\sigma_1}))\right)
\end{split}
\end{equation}
and symmetrically
\begin{eqnarray*}
Q_\chi^\delta:&(\C^{\gamma_2})_\chi&\longrightarrow (\C^V)_\chi\\
&f&\longmapsto \left(x\mapsto (Q_\chi^\delta f)(x)=\E^x_\delta(f(X_{\sigma_2}))\right)
\end{eqnarray*}
where $\sigma_i$ is the first hitting time of $\gamma_i$, $i=1,2$. Remark that, since $f$ is multivalued, $f(X_{\sigma_i})$ depends not only on the exit position $X_{\sigma_i}$ but also on the ``horizontal" winding number of the path up to $\sigma_i$.

Let us consider the operator $\phi_\chi:(\C^V)_\chi\rightarrow(\C^V)_\chi$ given by
$$f\mapsto R^\delta_\chi f_{|\gamma_1}+ Q^\delta_\chi f_{|\gamma_2}+f_{|V\setminus\gamma}.$$
It has a block decomposition
$$\phi_\chi=\left(   \begin{matrix} %
    \Id&(R^\delta_\chi)_{|\gamma_2}&0\\
    (Q^\delta_\chi)_{|\gamma_1}&\Id&0\\
    \ast&\ast&\Id
   \end{matrix}\right)$$
corresponding to the partition $V=\gamma_1\sqcup\gamma_2\sqcup(V\setminus\gamma)$.

Let us denote by $N_1$ the ``Neumann-jump operator" on $(\C^{\gamma_1})_\chi$ given by
$$N_1 f=(\Lap_\chi^\delta (R^\delta_\chi f))_{|\gamma_1}$$
and $N_2$ operating on $(\C^{\gamma_2})_\chi$ is defined similarly. Then
$$\Lap_\chi^\delta\phi_\chi=\left(   \begin{matrix} %
   N_1&0&\ast\\
   0&N_2&\ast\\
   0&0&(\Lap^\delta_\chi)_{|V\setminus\gamma}
   \end{matrix}\right).$$
The bottom right last block corresponds to a Laplacian with Dirichlet conditions on $V\setminus\gamma$ and depends on $\chi$ only through $\chi(A)$ and not $\chi(B)$.

Finally, if $\psi_\gamma$ is the operator on $(\C^V)_\chi$ given by
$$\psi_\gamma f=(R_\chi^\delta f_{|\gamma_1})+f_{|V\setminus\gamma_1}$$
we have
\begin{align*}
\psi_\chi&=\left(   \begin{matrix} %
  \Id&0\\
   \ast&\Id
   \end{matrix}\right)\\
\Lap_\chi^\delta\psi_\chi&=\left(   \begin{matrix} %
   N_1&\ast\\
   0&(\Lap_\chi^\delta)_{|V\setminus\gamma_1}
   \end{matrix}\right)
   \end{align*}
where the bottom right block does not depend on $\chi(B)$, in the block decomposition corresponding to the partition $V=\gamma_1\sqcup(V\setminus\gamma_1)$. We conclude that
$$\frac{\det(\Lap_\chi^\delta)}{\det(N_i)}$$
does not depend on $\chi(B)$ (for $i=1,2$), and neither does
$$\frac{\det(\Lap_\chi^\delta\phi_\chi)}{\det(N_1)\det(N_2)}.$$
Together with the classical identity
$$\det(\Id-BA)=\det\left( \begin{matrix} %
  \Id&A\\
   B&\Id
   \end{matrix}\right)=\det(\Id-AB)$$
this yields the

\begin{lem}\label{Lem:Fred}
If $S_\chi^\delta$ is the operator on $(\C^{\gamma_1})_\chi$ given by
$$S_\chi^\delta f=(Q_\chi^\delta((R_\chi ^\delta f)_{|\gamma_2}))_{|\gamma_1}$$
then
$$\frac{\det(\Lap_{\chi'}^\delta)}{\det(\Lap_{\chi}^\delta)}=\frac{\det(\Id-S_{\chi'}^\delta)}{\det(\Id-S_\chi^\delta)}$$
whenever $\chi$ is non-trivial and $\chi(A)=\chi'(A)$.
\end{lem}
Concretely, if $\sigma_{21}$ is the time of the first visit to $\gamma_1$ after the first visit to $\gamma_2$,
\begin{equation}\label{eq:Schidef}
(S_\chi^\delta f)(x)=\E^x_\delta(f(X_{\sigma_{21}}))
\end{equation}
where $f$ is understood as a $\chi$-multivalued function on the universal cover.

Let us sketch an alternative proof of Lemma \ref{Lem:Fred}, closer to Proposition 2.2 in \cite{Dub_SLEGFF}. Let $c_\chi(xy)=c(xy)$ for a generic edge $(xy)$ of $\Gamma$, and $c_\chi(xy)=\chi(B)c(xy)$ if $(xy)$ traverses an $A$-cycle bounding a fundamental domain with direct orientation, and likewise for other edges crossing the boundary of a fundamental domain, so that $\Lap_\chi$ may be identified with $(c_\chi(xy))_{x,y\in V}$. Then we have the expansion
$$-\log\det(\Lap^\delta_\chi)=\lim_{\eps\searrow 0}\sum_{k\geq 1}\frac{(1-\eps)^k}{k}\sum_{v_0,\dots,v_{k-1}\in V}c_\chi(v_0v_1)\dots c_\chi(v_{k-1}v_0).$$
Then one can decompose each summand (corresponding to a rooted loop on $\Gamma$) between successive visits to $\gamma_1$ and $\gamma_2$; some rather tedious bookkeeping leads to Lemma \ref{Lem:Fred}.

\subsection{Contractions}

Given Lemma \ref{Lem:Fred}, we are now concerned with the convergence of
$$\det(\Id-S_\chi^\delta)$$
along a suitable sequence of graphs $(\Gamma_\delta)_\delta$ on $\Sigma$ (the mesh $\delta$ going to zero along some sequence), where $S_\chi^\delta$ is given by \eqref{eq:Schidef}. We assume that $\gamma_1=\gamma_1^\delta$ (resp. $\gamma_2$) is a simple cycle on $\Gamma_\delta$ within $o(1)$ of $[0,1]$ (resp. $\frac\tau 2+ [0,1]$), in the sense of uniform convergence up to reparametrization.

From \eqref{eq:Schidef}, it is obvious that $\vertiii{S_\chi^\delta}_\infty\leq 1$. Here $\vertiii{.}_\infty$ denotes the $L^\infty$ operator norm. From the maximum principle, one can also argue that $\vertiii{S_\chi^\delta}_\infty<1$ for fixed mesh $\delta$. In order to control expansions of $\det(\Id-S_\chi^\delta)$ in the small mesh limit, we need an operator norm estimate uniform 
in $\delta$.

\begin{lem}\label{Lem:contract}
Fix $\chi:\pi_1(\Sigma)\rightarrow\U$ non-trivial. There is $\eps=\eps(\chi)>0$ such that for $\delta$ small enough, $S_\chi^\delta:(\C^{\gamma_1^\delta})_\chi\rightarrow (\C^{\gamma_1^\delta})_\chi$ is a $(1-\eps)$-contraction:
$$\vertiii{S_\chi^\delta}_\infty\leq 1-\eps.$$
\end{lem}

\begin{proof}
For simplicity of exposition we treat the case $\chi(B)\neq 1$, $\tau$ pure imaginary, the general case being similar. Let $f\in L^\infty(\gamma_1)$ with $\|f\|_\infty\leq 1$. For $x\in\gamma_1$, we have
$$(S_\chi^\delta f)(x)=\E^x_\delta(f(X_{\sigma_{21}}))$$
and want to show $|S_\chi^\delta f(x)|\leq 1-\eps$. Take a rectangle $R=[-\frac 15,\frac 15]\times [\frac{\Im\tau} 4,\Im\tau]$. By the harmonic measure estimate of Lemma \ref{Lem:harmmeas}, the probability that the RW started from $\tau/2$ exits $R$ on the top side is bounded away from zero for small $\delta$. 

For any $x\in\gamma_1$ (say with $\Re(x)\in[\frac 25,\frac 35]$), one can find (again by Lemma \ref{Lem:harmmeas}) a polygonal ($L$-shaped) domain $U_x$ with a boundary arc $J_x$ s.t. : 
\begin{enumerate}
\item $B(x,\eta)\subset U_x\subset B(x,\eta^{-1})$ for some positive $\eta$ independent of $x$; 
\item $\Harm_\delta(x,J_x,U_x)$ is bounded away from zero (uniformly in $x$ and in $\delta$ small enough); 
\item any path started from $x$ exiting $U_x$ on $J_x$ intersects any path started from $\tau/2$ exiting $R$ on top. 
\end{enumerate}

Similarly, one can find a polygonal domain $U'_x$ with a boundary arc $J'_x$ s.t. any path started from $x$ exiting $U'_x$ on $J'_x$ intersects any path started from $-\tau/2$ exiting $R-\tau$ on top (also satisfying (1)-(2)), and crosses $-\frac\tau 2+[0,1]$. 

One can sample a RW started from $x$ as follows. First sample a RW $Y$ started from $\tau$ and stopped when exiting $R$. Then run a RW $X$ started from $x$ up to first intersection of $Y$ or $Y-\tau$; then follow $Y$ or $Y-\tau$. This shows that there are two events $E,E'$ with probability bounded away from zero and a measure preserving correspondence $\omega\in E\mapsto \omega'\in E'$ s.t. $X_{\sigma_{21}}(\omega)=X_{\sigma_{21}}(\omega')+\tau$ (on the universal cover). See Figure \ref{Fig:contract}.
\begin{figure}[htb]
\begin{center}
\leavevmode
\includegraphics[width=1.0\textwidth]{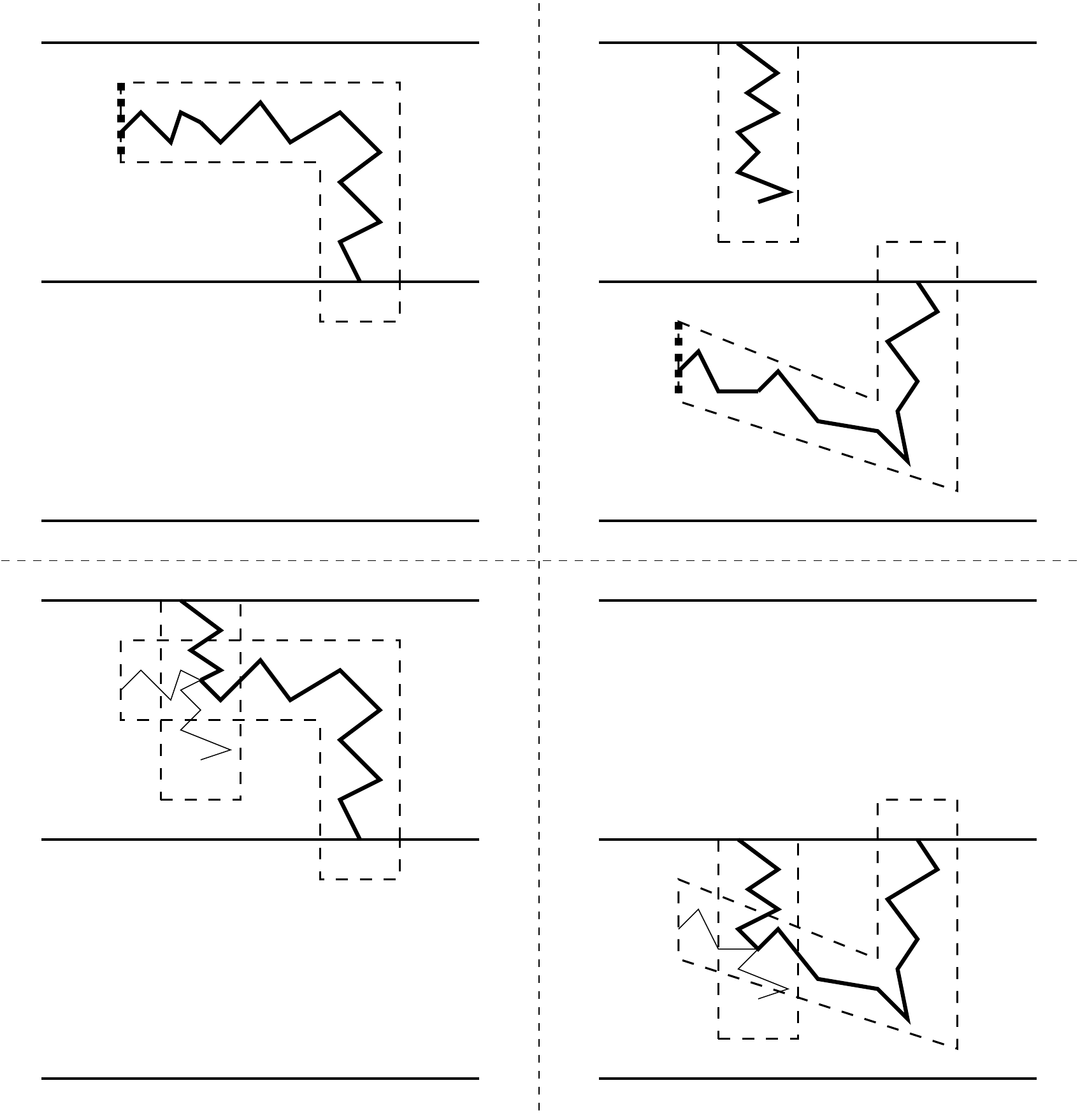}
\end{center}
\caption{Top left: a RW (thick) in $U_x$ (dashed) exiting on $J_x$ (thick, dashed). Top right: a RW $Y$ in $R$ exiting on top; and a RW in $U'_x$ (dashed) exiting on $J'_x$ (thick, dashed). Bottom: coupled random walks (thick) with the same terminal segment up to translation by $\tau$.}
\label{Fig:contract}
\end{figure}

Thus
$$|(S_\chi^\delta f)(x)|\leq \P_x(E)|\chi(B)+1|+1-2\P_x(E)\leq 1-\eps$$
as claimed (since $|\chi(B)+1|<2$).

\end{proof}

Alternatively, one could reason by contradiction, by extracting a subsequence $f_\delta\in L^\infty(\gamma_1^\delta)$ (the mesh $\delta$ goes to zero along some subsequence) with $\|f_\delta\|_\infty\leq 1$, $\|S_\chi f_\delta\|_\infty\nearrow 1$, and then a uniformly convergent subsequence; for this we need a Harnack estimate in lieu of Lemma \ref{Lem:harmmeas}.

\subsection{Iterated traces}

We now turn to the convergence of iterated traces $\Tr((S_\chi^{\delta})^k)$ as $\delta\searrow 0$, for $k\geq 1$ fixed. At some general level, $S_\chi^\delta$ is built from discrete harmonic measure, which converges (Lemma \ref{Lem:harmmeas}) to (continuous) harmonic measure. We proceed to check that this is enough to ensure convergence of the iterated traces to a universal limit. 

In the continuum, we can define a natural limiting operator (compare with \eqref{eq:Schidef}) by
\begin{equation}\label{eq:Schicont}
(S_\chi f)(x)=\E^x(f(B_{\sigma_{21}}))
\end{equation}
which can be decomposed as $S_\chi=Q_\chi R_\chi$ with (compare with \eqref{eq:Rchidef})
\begin{align*}
R_\chi:L^\infty_\chi(\gamma_1)\rightarrow L^\infty_\chi(\gamma_2):&&(R_\chi f)(x)=\E^x(f(B_{\sigma_1})),\\
Q_\chi:L^\infty_\chi(\gamma_2)\rightarrow L^\infty_\chi(\gamma_1):&&(Q_\chi f)(x)=\E^x(f(B_{\sigma_2})),
\end{align*}
where we take $\gamma_1=[0,1]$, $\gamma_2=\frac\tau 2+[0,1]$ for concreteness. Here $L^\infty_\chi(\gamma_i)$ designates measurable bounded $\chi$-multivalued functions on $\gamma_i$ with natural $L^\infty$ norm ($\chi$ is unitary).

We may also write it as an integral kernel operator
$$(R_\chi f)(x)=\int_{\gamma_1}f(y)R_\chi(x,y)dy$$
and likewise for $Q_\chi$. Plainly, $R_\chi$ (as a kernel) is bicontinuous and $S_\chi$ is trace-class.

From Lemma \ref{Lem:harmmeas}, it is easy to see that $R^{\delta}_\chi$ converges to $R_\chi$ in the following weak sense: if $f$ continuous on $\gamma_1$, and $f_\delta$ the restriction to $\gamma_1^\delta$ of a continuous extension of $f$, then $R_\chi^{\delta}f_\delta$
converges to $R_\chi f$ pointwise on $\gamma_2$. In order to get convergence of the trace we need additional continuity estimates (of the discrete harmonic measure w.r.t. the starting point).

\begin{lem}\label{Lem:trace}
If $\gamma_i^\delta$ converges to $\gamma_i$, $i=1,2$, and $k\geq 1$, then
$$\Tr((S^{\delta}_\chi)^k)\stackrel{\delta\searrow 0}{\longrightarrow}\Tr((S_\chi)^k)$$
\end{lem}
\begin{proof}
We treat the case $k=1$, the general case being similar. The discrete trace may be written as
$$\Tr(S^{\delta}_\chi)=\sum_{x\in\gamma_1^\delta}\sum_{y\in\gamma_2^\delta}Q^{\delta}_\chi(x,y)R^{\delta}_\chi(y,x)$$
and the continuous trace (in the sense of trace-class operators, see e.g. \cite{Simtrace}) is
$$\Tr(S_\chi)=\int_{\gamma_1}\int_{\gamma_2}Q_\chi(x,y)R_\chi(y,x)dxdy$$
Here we look at $\gamma_i$ as closed cycles on $\Sigma$, $i=1,2$. Let $U(a,b)\subset\C$ denote the horizontal strip
$$U(a,b)=\{z\in\C: a<\Im z<b\}$$
and $t=\frac{\Im\tau}2$.

If we lift the cycles $\gamma_i$ to the universal cover, we get the expression
\begin{align*}
\Tr(S_\chi)&=\sum_{m\in \Z}\int_{[0,1]}\int_{\frac\tau 2+[0,1]}\chi(A)^m\Harm(x-m,dy,U(-t,t))\Harm(y,dx,U(0,2t))\\
&+\sum_{m\in \Z}\int_{[0,1]}\int_{-\frac\tau 2+[0,1]}\chi(A)^m\Harm(x-m,dy,U(-t,t))\Harm(y,dx,U(-2t,0))\\
&+\chi(B)\sum_{m\in \Z}\int_{\tau+[0,1]}\int_{\frac\tau 2+[0,1]}\chi(A)^m\Harm(x-m,dy,U(-t,t))\Harm(y,dx,U(0,2t))\\
&+\overline{\chi(B)}\sum_{m\in \Z}\int_{-\tau+[0,1]}\int_{-\frac\tau 2+[0,1]}\chi(A)^m\Harm(x-m,dy,U(-t,t))\Harm(y,dx,U(-2t,0))\end{align*}
so that everything is expressed in terms of harmonic measure in strips (and likewise for $\Tr(S_\chi^\delta)$). In terms of the lift of the Brownian motion to $\C$ started at $B_0$ with $\Im B_0=0$, this corresponds to the four possibilities:
$$(\Im B_{\sigma_2},\Im B_{\sigma_{21}})=(t,0),(-t,0),(t,2t),{\rm\ or\ }(-t,-2t).$$
Up to a small (uniformly in $\delta$) error, one may replace the $\sum_{m\in\Z}$ by finite truncated sums and the strips by long rectangles.

We are left with proving convergence of a term of the type 
$$\sum_{x\in\gamma^\delta_1}\sum_{y\in\gamma^\delta_2}\Harm_\delta(x-m,\{y\},R)\Harm_\delta(y,\{x\},R+it)$$
where $R$ is a long rectangle and $m\in\Z$ is fixed. For simplicity of notation set $m=0$. From the convergence of Poisson kernels (Lemma 1.2 of \cite{yadin_lerw}), it follows that for any $\eps>0$, there is $L>0$ such that
\begin{equation}\label{eq:harnackest}
\left|\frac{\Harm_\delta(x',\{y\},R)}{\Harm_\delta(x,\{y\},R)}-1\right|\leq\eps
\end{equation}
for all $x,x'\in\gamma_1$ with $|x-x'|\leq L^{-1}$, $y\in\partial R$ and $\delta$ small enough. (By contrast, notice that $y\mapsto\log\Harm_\delta(x,\{y\},R)$ is typically highly oscillatory). Let us partition $\gamma_1^\delta$, $\gamma_2^\delta$ into $O(L)$ intervals $I^\delta_i$, $J^\delta_j$ of diameter $\leq L^{-1}$; and pick a point $x^\delta_i$, $y^\delta_j$ in any such interval. We may assume $(I^\delta_i)_i$ converges to a partition $(I_i)_i$ of $\gamma_1$, and similarly for $(J^\delta_j)$; we also assume that $x_i^\delta\rightarrow x_i$, $y_j^\delta\rightarrow y_j$ as $\delta\searrow 0$. 
Then
\begin{align*}
\sum_{x\in\gamma^\delta_1}\sum_{y\in\gamma^\delta_2}\Harm_\delta(x-m,\{y\},R)\Harm_\delta(y,\{x\},R+it)\hphantom{============}\\
=\left(\sum_{i,j}\Harm_\delta(x_i^\delta,J_j^\delta,R)\Harm_\delta(y_j^\delta,I^\delta_i,R+it)\right)(1+O(\eps)).
\end{align*}
By Lemma \ref{Lem:harmmeas}, the sum in the RHS converges (for fixed $L$, as $\delta\searrow 0$) to 
$$\sum_{i,j}\Harm(x_i,J_j,R)\Harm(y_j,I_i,R+it).$$
Moreover, we can pick $L$ large enough so that this sum is within $\eps$ of
$$\int_{[0,1]}\int_{\frac\tau 2+[0,1]}\Harm(x,dy,R)\Harm(y,dx,R+it)$$
by Riemann sum approximation and the Harnack estimate \eqref{eq:harnackest}. For general $k$, one expresses $\Tr(S_\chi^k)$ as a sum of $4^k$ $(2k)$-fold integrals. This iterated trace corresponds to loops traveling $k$ times back and forth between $\gamma_1$ and $\gamma_2$; each traversal from $\gamma_1$ to $\gamma_2$ (or vice versa) is either upward or downward. The convergence for any fixed $k$ is handled as in the case $k=1$, which concludes the argument.
\end{proof}

\subsection{Conclusion}

We may now complete the proof of Proposition \ref{Prop:detconv}.
\begin{proof}[Proof of Proposition \ref{Prop:detconv}.]
By Lemma \ref{Lem:Fred}, 
$$\frac{\det\Lap^{\delta}_{\chi'}}{\det\Lap^{\delta}_\chi}=\frac{\det(\Id-S_{\chi'}^{\delta})}{\det(\Id-S_\chi^{\delta})}$$
If $\chi(A)=\chi'(A)$. 

Lemma \ref{Lem:contract} justifies the expansion
\begin{equation}\label{eq:propdetconv}
-\log\det(\Id-S_\chi^{\delta})=\sum_{k\geq 1}\frac 1k\Tr((S_\chi^{\delta})^k)
\end{equation}
where the sum converges exponentially fast.

Consider $S_\chi^{\delta}$ as an operator $L^1_\chi(\gamma_1^\delta,\mu_\delta)\rightarrow L^\infty_\chi(\gamma_1^\delta)$, where $\mu_\delta$ is the harmonic measure on $\gamma_1^\delta$ seen from a reference point on $\gamma_2^\delta$, say. From Lemma 1.2 in \cite{yadin_lerw}, we see that for any $x,x'\in\gamma_2^\delta,y\in\gamma_1^\delta$
$$\frac{\Harm_\delta(x',\{y\},\Gamma_\delta\setminus\gamma_1^\delta)}{\Harm_\delta(x,\{y\},\Gamma_\delta\setminus\gamma_1^\delta)}\leq C$$
for some constant $C$ independent of $\delta$. It follows that $\vertiii{S_\chi^{\delta}}_{L^1\rightarrow L^\infty}\leq C$, and that (Lemma \ref{Lem:contract}) $\vertiii{(S_\chi^{\delta})^k}_{L^1\rightarrow L^\infty}$ decays exponentially uniformly in $\delta$.

It is easy to check that for any (finite-dimensional) operator $S$, 
$$|\Tr(S)|\leq\vertiii{S}_{L^1(\mu)\rightarrow L^\infty}\|\mu\|_{TV}$$ 
($\|.\|_{TV}$ denotes the total variation). Indeed,
\begin{align*}
|\Tr(S)|&\leq\sum_{i=1}^d |S_{ii}|\leq\sum_{i=1}^d\|Se_i\|_\infty\\
&\leq\sum_{i=1}^d\vertiii{S}_{L^1(\mu)\rightarrow L^\infty}\|e_i\|_{L^1(\mu)}
\leq\vertiii{S}_{L^1(\mu)\rightarrow L^\infty}\|\mu\|_{TV}
\end{align*}
where $d$ is the dimension of $S$.

Consequently, the convergence of the sum in \eqref{eq:propdetconv} is uniform in $\delta$. 
Lemma \ref{Lem:trace} gives term-wise convergence. This gives the result with the condition $\chi(A)=\chi'(A)$ (e.g. upon setting $h(\chi)=\det(\Id-S_\chi)/\det(\Id-S_{{\chi _0}})$ for a reference character $\chi_0$). Applying this twice (with an affine change of coordinates or equivalenty a change of homology basis $(A,B)\rightarrow (B,-A)$ to exchange the roles of the $A$ and $B$ cycles) gives the general case.
\end{proof}

\section{The case of the square lattice}\label{Sec:square}

From Proposition \ref{Prop:detconv} we know that, in order to identify the limit, it is enough to work with any specific sequence of graphs of our choosing. For this, one could use the results of \cite{BdT_loop} (on the hexagonal lattice and for $\Re\tau=0$), or of \cite{KSW_torus} (Theorem 3, for general $\tau$ and bipartite periodic graphs); these are based on spectral arguments, i.e. on estimating the eigenvalues of the Kasteleyn matrix twisted by a character. Alternatively, one could use Theorem 7 \cite{Dub_tors} (Section 5, for general $\tau$ and isoradial graphs), which follows from variational arguments (viz. estimating the variation of the determinant of a twisted Kasteleyn matrix w.r.t. the characters). Another possible approach, involving more machinery, is based on relating the Fredholm determinant $\det(\Id-S_\chi)$ (where $S_\chi$ is the trace-class operator of \eqref{eq:Schicont}) to $\zeta$-regularized determinants and using \cite{RS_analytic}.

For the sake of variety and self-containedness, we give (yet) another, shorter proof for the square lattice, based on heat kernel convergence; remark that such estimates, at fixed (microscopic) location and fixed time, are not available under the general assumption \eqref{eq:weakconv}.

So let us take $\delta^{-1}\in\N$, and
$$\Gamma_\delta=\left(\delta\Z^2\right)/(\Z+\tau_\delta\Z)$$
where $\tau_\delta=\tau+o(1)$ (one could also fix $\tau_\delta$ and apply a small affine distortion to the square lattice; we shall omit the dependence of $\tau$ on $\delta$ from now on). Let $\Lambda$ denote the lattice $(\Z+\tau\Z)\simeq\pi_1(\Sigma)$ (so that $(1,\tau)$ corresponds to $([A],[B])$). The weights are, say, $\frac 14$ on all edges (nearest neighbors), so that the corresponding RW is a simple random walk (SRW) and $\Lap$ is the discrete Laplacian:
$$(\Lap^{\delta} f)(z)=f(z)-\frac 14\left(f(z+\delta)+f(z-\delta)+f(z+i\delta)+f(z-i\delta)\right).$$ 
Let $\chi:\pi_1(\Sigma)\rightarrow\U$ a non-trivial character. Then $\Lap_\chi^{\delta}$ has a complete set of eigenvectors of the type
$$f(x+iy)=\exp(i(\alpha x+\beta y))$$
where $e^{i\alpha}=\chi(1)$, $e^{i(\alpha\Re\tau+\beta\Im\tau)}=\chi(\tau)$. It follows that, by writing $\Lap^\delta_\chi=\Id-P^\delta_\chi$, the eigenvalues of $P^\delta_\chi$ have modulus $<1$ (one can also obtain $\vertiii{P_\chi^\delta}_\infty<1$ using the maximum principle). This justifies the expansion
\begin{equation}\label{eq:trexp}
-\log\det(\Lap_\chi^{\delta})=\sum_{k\geq 1}\frac 1k\Tr((P^{\delta}_\chi)^k)
\end{equation}
(with convergence for fixed $\delta$ but not uniformly in $\delta$). Concretely, if $x,y$ are in a fundamental domain and $P^{\delta}$ denotes the transition matrix for the SRW  on $\delta\Z^2$, we have
$$(P^{\delta}_\chi)^k(x,y)=\sum_{\ell\in\Lambda}\chi(\ell)(P^{\delta})^k(x,y+\ell).$$
We shall need a few standard estimates on $P^\delta$. 
Let $t=k/n^2$, $(P_t)_{t\geq 0}$ be the heat kernel, the semigroup of standard Brownian motion on $\C$ and $p_t$ its density:
$$p_t(x,y)=\frac{1}{2\pi t}\exp(-\frac{|y-x|^2}{2t})$$
With $\delta=1/n$ and when $n$ is even (we take $2k$ and $n$ even in order to avoid essentially notational issues due to parity), we have the Local Central Limit theorem
\begin{equation}\label{eq:LCLT}
\begin{split}
(P^{\delta})^{2k}(x,y)&=\frac 2{n^2}p_t(x,y)(1+o(1)){\rm\ \ \ \ \ if\ }n|y-x|=o(k^{3/4})\\
&=\frac 2{n^2}p_t(x,y)(1+O(1)){\rm\ \ \ \ \  if\ }n|y-x|=O(k^{3/4})
\end{split}
\end{equation}
which is valid in a diffusive (i.e. $|y-x|=O(\sqrt t)$) and up to a moderate deviation regime. See e.g. Proposition 2.5.3 in \cite{LawLim} (simply using Stirling's formula and the usual trick of projecting the $2d$ SRW $(X,Y)$ on the diagonals to obtain independent 1d SRWs $X+Y$ and $X-Y$).

At larger scales, one can use the following large deviation estimate
\begin{equation}\label{eq:ldev}
(P^{\delta})^{2k}(x,y)=O(\exp(-c(|x-y|n)^2/k))
\end{equation}
which follows e.g. from the exponential Chebychev inequality ($c$ is a positive constant).

Finally, we shall need a first difference estimate (see e.g. Theorem 2.3.6 in \cite{LawLim}):
\begin{equation}\label{eq:firstdiff}
(P^{\delta})^{2k}(x,y+1)-(P^{\delta})^{2k}(x,y)=O(n^{-2}t^{-3/2}).
\end{equation}

From \eqref{eq:trexp}, we may write
\begin{equation}\label{eq:trexp2}
\begin{split}
-\log\frac{\det(\Lap_{\chi'}^{\delta})}{\det(\Lap_\chi^{\delta})}&=\sum_{k\geq 1}\frac 1k\left(\Tr((P^{\delta}_\chi)^k)-\Tr((P^{\delta}_{\chi'})^k)\right)\\
&=\sum_{k\geq 1}\frac 1k\sum_{x\in\Sigma}\sum_{\ell\in\Lambda}\left(\chi(\ell)-\chi'(\ell)\right)(P^{\delta})^k(x,x+\ell)\\
&=\sum_{k\geq 1}\frac 1{2k}\sum_{x\in\Sigma}\sum_{\ell\in\Lambda}\left(\chi(\ell)-\chi'(\ell)\right)(P^{\delta})^{2k}(x,x+\ell).
\end{split}
\end{equation}
Remark that $\chi(\ell)-\chi'(\ell)=0$ when $\ell=0$, so that the summand vanishes for $k=o(n)$. The last line follows from assuming $n$ even, $n\tau\in2(\Z+i\Z)$.

We can discuss convergence of the sum in \eqref{eq:trexp2} as $n\rightarrow\infty$ for small, intermediate, and long times. 
\begin{enumerate}
\item
For short times, the large deviation estimate \eqref{eq:ldev} guarantees that
 $$\sum_{k\leq \lfloor n^{3/2}\rfloor}(\dots)=\sum_{k\leq \lfloor n^{3/2}\rfloor}O(k^{-1}n^2\exp(-c'n^2/k))$$
which goes to zero as $n\rightarrow\infty$ ($c'>0$ is a positive constant).
 
\item
For intermediate times, say $\lfloor n^{3/2}\rfloor<k\leq \lfloor Tn^{2}\rfloor$ with $T$ large but independent of $n$, the random walk quantities converge to their natural continuous counterpart by the LCLT \eqref{eq:LCLT}:
$$\sum_{\ell\in\Lambda}\left(\chi(\ell)-\chi'(\ell)\right)(P^{\delta})^{2k}(x,x+\ell)=\frac{2}{n^2}\sum_{\ell\in\Lambda}(\chi(\ell)-\chi'(\ell))\frac{e^{-|\ell |^2/2t}}{2\pi t}+o(1)$$
where $t=2k/n^2$. (More precisely, we use the LCLT for $n|\ell|=o(k^{3/4})$ and the large deviation estimate \eqref{eq:ldev} otherwise).

Remark that $|\Gamma_\delta|\sim n^2\Im\tau$. It follows that
$$\sum_{\lfloor n^{3/2}\rfloor<k\leq \lfloor Tn^{2}\rfloor}(\dots)=\Im\tau\int_0^T\left(\sum_{\ell\in\Lambda}(\chi(\ell)-\chi'(\ell))e^{-|\ell |^2/2t}\right)\frac{dt}{2\pi t^2}+o(1).$$
The convergence of the integral as $T\rightarrow\infty$ is not immediately apparent but can be justified using e.g. the Poisson summation formula, see Section 4 of~\cite{RS_analytic}.

\item
Finally we need to control $\sum_{k>\lfloor Tn^{2}\rfloor}(\dots)$ uniformly in $n$. For such a $k$, let $B_k$ (resp. $B'_k$) be a box of diameter $O(t^{1/2+\eps})$ (resp. $O(n^{-1}k^{3/4})$) around the origin, so that $B_k\subset B'_k$. 

Inside $\Lambda\cap B_k$ (which contains $O(t^{1+2\eps})$ points), we use an Abel summation by part argument. Assume $\eta=\chi(1)\neq 1$ (otherwise reason on $\chi(\tau)$). Then
$$\chi(\ell)=(\eta-1)^{-1}(\chi(\ell+1)-\chi(\ell))$$
so that
\begin{align*}
\sum_{\Lambda\cap B_k}\chi(\ell)(P^\delta)^{2k}(x,x+\ell)&=(1-\eta)^{-1}\sum_{\ell\in\Lambda\cap B_k}\chi(\ell)\left((P^{\delta})^{2k}(x,y+\ell+1)-(P^{\delta})^{2k}(x,y+\ell)\right)\\
&+({\rm boundary\ terms}).
\end{align*}
Then from \eqref{eq:firstdiff}, we obtain
$$\sum_{\Lambda\cap B_k}\chi(\ell)(P^\delta)^{2k}(x,x+\ell)=O(t^{1+2\eps}n^{-2}t^{-3/2})$$
(the much smaller contribution of boundary terms is handled by \eqref{eq:ldev}). 

In $B'_k\setminus B_k$, we use the LCLT \eqref{eq:LCLT} to obtain 
\begin{align*}
\sum_{\Lambda\cap (B'_k\setminus B_k)}\chi(\ell)(P^\delta)^{2k}(x,x+\ell)&=O\left(\sum_{\lfloor t^{1/2+\eps}\rfloor\leq m\leq \lfloor\frac{k^{3/4}}n\rfloor}
\frac{m}k\exp(-c'm^2/t))
\right)\\
&=O(n^{-2}\exp(-c't^{2\eps}))
\end{align*}
for some $c'>0$ (depending on $\tau$); as is customary, the exact value of $c'$ may change from line to line.

Finally, outside $B'_k$ we have by \eqref{eq:ldev}
$$\sum_{\Lambda\cap (B'_k)^c}\chi(\ell)(P^\delta)^{2k}(x,x+\ell)=O\left(\sum_{m\geq \lfloor \frac{k^{3/4}}n\rfloor} m\exp(-c'm^2/t)\right)=O(t\exp(-c'\sqrt k))$$
for some $c'>0$.

Consequently,
$$\frac{1}{2k}\sum_{x\in\Sigma}\sum_{\Lambda}(\dots)=O(t^{2\eps-1/2})+O(\exp(-c't^{2\eps})+O(k\exp(-c'\sqrt k))$$
and
\begin{align*}
\sum_{k>\lfloor Tn^{2}\rfloor}(\dots)&=O\left(\sum_{k>\lfloor Tn^{2}\rfloor}k^{-1}(kn^{-2})^{2\eps-1/2}\right)
+O\left(\sum_{k>\lfloor Tn^{2}\rfloor}\exp(-c'\sqrt k)\right)\\
&=O(T^{2\eps-\frac 12})+o(1)
\end{align*}
with $o(1)$ going to zero as $Tn^2$ goes to infinity.
\end{enumerate}

In conclusion we have
$$\frac{\det(\Lap_{\chi'}^{\delta})}{\det(\Lap_\chi^{\delta})}\longrightarrow\exp\left(-\Im\tau\int_0^\infty\frac{dt}{2\pi t^2}\sum_{\ell\in\Lambda}(\chi(\ell)-\chi'(\ell))e^{-|\ell |^2/2t}\right).$$

This is a classical quantity, which is evaluated in Theorem 4.1 of \cite{RS_analytic} (as a reformulation of Kronecker's second limit formula in analytic number theory), as we now explain.

Let $T(\chi)$ be the {\em analytic torsion} of the unitary bundle associated with the character $\chi$ (\cite{RS_analytic}); for our purposes we can take as definition of $T$, up to multiplicative constant,
\begin{equation}\label{eq:deftors}
\begin{split}
\frac{T(\chi')}{T(\chi)}&=\exp\left(-\frac 12\int_0^\infty\frac{dt}t\int_\Sigma\sum_{\ell\in\Lambda}(\chi'(\ell)-\chi(\ell)) p_t(x,x+\ell)dA(x)\right)\\
&=\exp\left(-\frac{\Im\tau}2\int_0^\infty\frac{dt}{2\pi t^2}\sum_{\ell\in\Lambda}(\chi'(\ell)-\chi(\ell))e^{-|\ell |^2/2t}\right)
\end{split}
\end{equation}
(here $dA$ represents the area measure on the torus) and we have established
$$\frac{\det(\Lap_{\chi'}^{\delta})}{\det(\Lap_\chi^{\delta})}\longrightarrow\frac{T(\chi')^2}{T(\chi)^2}.$$
Remark that $T(\chi)^2$ is also, by definition \cite{RS_analytic}, the {\em $\zeta$-regularized determinant} of $\Lap_\chi$ (the continuous Laplacian operating on $\chi$-multivalued function); the argument above thus establishes the convergence of determinants of discrete Laplacians to $\zeta$-regularized determinants of continuous Laplacians (in a relative or projective sense).

Let $u,v\in [0,1)$ be s.t. $\chi(m\tau+n)=\exp(2i\pi(mu+nv))$. Then Theorem 4.1 of \cite{RS_analytic} provides the following evaluation of $T(\chi)$:
$$T(\chi)=|\eta(\tau)^{-1}e^{i\pi v^2\tau}\theta(u-v\tau|\tau)|$$
where $\eta$ is the Dedekind $\eta$ function ($q=e^{i\pi\tau}$):
\begin{equation}\label{eq:eta}
\eta(\tau)=q^{1/12}\prod_{n=1}^\infty(1-q^{2n})
\end{equation}
and $\theta(.|\tau)$ is the odd $\theta$ function (conventions as in \cite{Chandra}):
\begin{equation}\label{eq:theta}
\begin{split}
\theta(w|\tau)&=-i\sum_{n=-\infty}^\infty(-1)^nq^{(n+\frac 12)^2}e^{(2n+1)i\pi w}\\
&=q^{1/6}\eta(q)2\sin(\pi w)\prod_{n=1}^\infty(1-q^{2n}e^{2i\pi w})(1-q^{2n}e^{-2i\pi w}).
\end{split}
\end{equation}

This completes the proof of Proposition \ref{Prop:square}.

\section{Identification of the limiting distribution}

Based on Propositions \ref{Prop:detconv} and \ref{Prop:square}, we can now identify the limiting distribution of $[\mf m]$, which encodes the winding of root cycles in pairs of dual CRSFs, or equivalently the height change of the dimer configuration along cycles. The argument is rather similar to Section 5.2 of \cite{Dub_tors} and Section 4.C of \cite{ABMNV}; however we reason here based on Temperley's bijection and Forman's formula, rather than Kasteleyn's determinantal enumeration of dimers as in \cite{Dub_tors}.

Our goal is to identify the limiting distribution of $[{\mf m}]$ in the small mesh limit (recall \eqref{eq:cyclehomol}).

Let $k\geq 1$, and $\gamma$ a primitive element of $H_1(\Sigma,\Z)$ (i.e. $[\gamma]=m[A]+n[B]$, $\gcd(m,n)=1$). As is well-known, any (non-contractible) simple cycle is primitive. Let $Z_{k_+,k_-,[\gamma]}$ be the partition function (i.e. the sum of the weights) of oriented CRSFs on $\Gamma$ with exactly $k_+\geq 0$ (resp. $k_-$) cycles with class $[\gamma]$ (resp. $-[\gamma]$), $\gamma$ a primitive cycle (so that $Z_{k_+,k_-,[\gamma]}=Z_{k_-,k_+,-[\gamma]}$). (In higher genus, root cycles of CRSFs need not be homologous).

Given $F$ a CRSF on $\Gamma$ counted in $Z_{k_+,k_-,\gamma}$, there are $2^{k_++k_-}$ dual CRSFs $F^*$ on $\Gamma^*$, corresponding to possible choices of orientations of the root cycles. Recall that by construction, the conductances on $\Gamma^*$ are all 1. 

Let ${\mc Z}_{k[\gamma]}$ denote the dimer partition function on ${\mc G}$ restricted to $[\mf m]=k[\gamma]$, i.e.
$${\mc Z}_{k[\gamma]}=\sum_{\mf m:[\mf m]=k[\gamma]}w({\mf m}).$$
Since dimer configurations are in measure-preserving bijection with pairs $(F,F^*)$, we have
\begin{equation}\label{eq:binom}
{\mc Z}_{k[\gamma]}=\sum_{k_+,k_-\geq 0}Z_{k_+,k_-,[\gamma]}{k_++k_- \choose k+k_-}\end{equation}
for $k\neq 0$: indeed, if $F^*$ has $k+k_-$ (resp. $k_+-k$) cycles in the class $[\gamma]$ (resp. $-[\gamma]$), then $[{\mf m}]=k[\gamma]$.

We may look at $k[\gamma]\mapsto {\mc Z}_{k[\gamma]}$ as a function on the abelian group $H_1(\Sigma,\Z)\simeq\Z^2$ (with finite support for a given $\Gamma$). Its Fourier-Pontryagin transform is: 
\begin{equation}\label{eq:fourpont}
\begin{split}
\hat{\mc  Z}(\rho)&={\mc Z}_0+\sum_{k>0,\gamma}\rho(\gamma)^k {\mc Z}_{k[\gamma]}\\
&=\frac 12\sum_{k_+,k_-\geq 0,[\gamma]} Z_{k_+,k_-,[\gamma]}(1+\rho(\gamma))^{k_++k_-}\rho(\gamma)^{-k_-}\\
&=\frac 12\sum_{k_+,k_-\geq 0,[\gamma]} Z_{k_+,k_-,[\gamma]}(1+\rho(\gamma))^{k_+}(1+\rho(\gamma)^{-1})^{k_-}
\end{split}\end{equation}
where $\rho:H_1(\Sigma,\Z)\rightarrow\U$ is a unitary character and the sums are over primitive classes $[\gamma]$; the second line follows from \eqref{eq:binom} and the binomial formula. 

In particular, if $1$ denotes the trivial character, $\hat {\mc Z}(1)={\mc Z}$ is the partition function of the model, and 
\begin{equation}\label{eq:charfun}
\frac{\hat {\mc Z}(\rho)}{\hat {\mc Z}(1)}=\E(\rho([{\mf m}]))
\end{equation}
is the characteristic function of $[{\mf m}]$ ($\E$ denotes the expectation under the dimer measure).

From Forman's formula (Theorem \ref{Thm:Forman}) we have
\begin{equation}\label{eq:Forman}
\det(\Lap_\chi)=\frac 12\sum_{k_+,k_-,[\gamma]}Z_{k_+,k_-,[\gamma]}(1-\chi(\gamma))^{k_+}(1-\chi(\gamma)^{-1})^{k_-}
\end{equation}

Observe that there are four characters $\epsilon$ of $H_1(\Sigma,\Z)$ with $\epsilon^2=1$ ($\epsilon([A])=\pm 1$ and $\epsilon([B])=\pm 1$). 
If $\chi$ is any character, $\sum_\epsilon(\epsilon\chi)(\gamma)=0$ unless $\gamma=2\gamma'$ for some $\gamma'$, in which case $\sum_\epsilon(\epsilon\chi)(\gamma)=4\chi(\gamma)$. Moreover, for $\gamma$ primitive (or, more generally, $[\gamma]\neq 0$ in $H_1(\Sigma,\Z/2\Z)$),
\begin{equation}\label{eq:signs}
-\chi^k(\gamma)+\frac 12\sum_{\epsilon:\epsilon^2=1}(\epsilon\chi)^k(\gamma)=\left\{\begin{array}{ll} \chi^k(\gamma)&{\rm\ if\ } k{\rm\ even}\\
-\chi^k(\gamma)&{\rm\ if\ } k{\rm\ odd.}\end{array}\right.
\end{equation}
The expressions \eqref{eq:fourpont}, \eqref{eq:Forman} show that the expansions of $\hat {\mc Z}(\chi)$, $\det(\Lap_\chi)$ in powers of $\chi(\gamma)$ agree on even powers and are opposite on odd powers. It follows that:
\begin{equation}\label{eq:spin}
\sum_\epsilon \hat {\mc Z}(\epsilon\chi)=\sum_\epsilon \det(\Lap_{\epsilon\chi})
\end{equation}
and that
$$\hat {\mc Z}(\chi)=-\det(\Lap_\chi)+\frac 12\sum_\epsilon\det(\Lap_{\epsilon\chi}).$$
Parameterizing characters by $\chi(r\tau+s)=\exp(2i\pi ru+2i\pi sv)$, $(u,v)\in[0,1)^2$, the convergence result of Proposition \ref{Prop:square} shows that
$$\det(\Lap_\chi)=c_\delta(1+o(1))\left(|e^{i\pi v^2\tau}\theta(v-u\tau|\tau)|^2\right)$$
as the mesh $\delta\searrow 0$, for $\chi$ non-trivial (where the positive sequence $c_\delta$ does not depend on $\chi$). Remark that $\det(\Lap_\chi)=0$ at $\chi=1$ since constant functions are then in the kernel. 
It follows that the characteristic function \eqref{eq:charfun} of $[{\mf m}]$ converges pointwise as the mesh $\delta$ goes to zero. We now want to identify the limiting distribution.

Let us compute the Fourier transform of $\chi\mapsto h(\chi)=|e^{i\pi v^2\tau}\theta(v-u\tau|\tau)|^2$. First we write
\begin{align*}
h(\chi)&=e^{-2\pi v^2\Im\tau}\sum_{m,n}(-1)^{m+n}q^{(m+\frac 12)^2}\bar q^{(n+\frac 12)^2}e^{i\pi(2m+1)w-i\pi(2n+1)\bar w}\\
&=\sum_{r,m}(-1)^re^{2i\pi r u}q^{(m+\frac 12)^2-v(2m+1)+v^2}\bar q^{(m-r+\frac 12)^2-v(2m-2r+1)+v^2}\\
&=\sum_{r,m}(-1)^re^{2i\pi r u}q^{(v-m-\frac 12)^2}\bar q^{(v-m+r-\frac 12)^2}
\end{align*}
where, as earlier, $q=e^{i\pi\tau}$. Applying the Poisson summation formula 
$$\sum_{m\in\Z}\phi(v+m)=\sum_{s\in\Z} e^{2i\pi sv}\hat\phi(s)$$
to the rapidly decaying function $\phi(v)=q^{(v-\frac 12)^2}\bar q^{(v+r-\frac 12)^2}=|q|^{2(v+\frac {r-1}2)^2+\frac{r^2}2}e^{-2i\pi\Re(\tau) r(v+\frac{r-1}2)}$, such that:
\begin{align*}
\hat\phi(s)&=\int_{-\infty}^\infty \phi(t)e^{-2i\pi st}dt%
=|q|^{\frac{r^2}2}\exp(-2i\pi s\frac{r-1}2)\int_{-\infty}^\infty \exp(-2\pi\Im\tau t^2-2i\pi(r\Re\tau+s)t)dt\\
&=|q|^{\frac{r^2}2}(-1)^{s(r-1)}\frac{1}{\sqrt{2\Im\tau}}\exp\left(-\frac\pi 2 (r\Re\tau+s)^2/\Im\tau\right)=(-1)^{s(r-1)}\frac{1}{\sqrt{2\Im\tau}}\exp\left(-\frac\pi {2\Im\tau} |r\tau+s|^2\right)
\end{align*}
one obtains:
$$h(\chi)=\sum_{r,s\in\Z}\chi(r\tau+s)(-1)^{(s-1)(r-1)+1}\frac{1}{\sqrt{2\Im\tau}}\exp\left(-\frac\pi {2\Im\tau} |r\tau+s|^2\right).$$
The sign in the summand is $-1$ unless $r,s$ are both even; remark that as in \eqref{eq:signs}
$$-1+\frac 12\sum_{\epsilon:\epsilon^2=1}\epsilon(r\tau+s)=\left\{\begin{array}{ll} 1&{\rm\ if\ } r=s=0\mod 2\\
-1&{\rm\ otherwise.}\end{array}\right.$$
Comparing with \eqref{eq:spin} gives
$$\lim_{\delta\searrow 0}c_\delta^{-1}\hat{\mc Z}(\chi)=|\eta(\tau)|^{-2}\sum_{r,s}\chi(r\tau+s)\frac{1}{\sqrt{2\Im\tau}}\exp\left(-\frac\pi {2\Im\tau} |r\tau+s|^2\right).$$
Pointwise convergence of the characteristic function \eqref{eq:charfun} yields convergence in law for $[{\mf m}]$. This concludes the proof of Theorem \ref{Thm:discrg}.

\bibliographystyle{abbrv}
\bibliography{bibliowind}

\end{document}